\documentclass[11pt]{amsart}
\usepackage{}
\usepackage{amssymb,amsmath,amsfonts,amsthm,here,enumerate,xfrac,mathtools,graphicx,pstool}
\usepackage[colorlinks]{hyperref}
\definecolor{cerulean}{rgb}{0,0.25,0.5}
\definecolor{uniblue}{rgb}{0.08,0.42,0.63}
\definecolor{mikeblue}{RGB}{0,117,234}
\hypersetup{
    linkcolor = {uniblue},
citecolor = {uniblue},
urlcolor = {uniblue}
}
\numberwithin{equation}{section}

\newcommand{\I}{\mathrm{I}}

\newcommand{\R}{\mathbb R}
\newcommand{\N}{\mathbb N}
\newcommand{\Z}{\mathbb Z}

\newtheoremstyle{plain}
  {10pt}
  {10pt}
  {\it}
  {0pt}
  {\bf}
  {}
  {\newline}
  {}
\newtheoremstyle{definition}
  {10pt}
  {10pt}
  {}
  {0pt}
  {\bf}
  {}
  {\newline}
  {}
\theoremstyle{plain}

\newtheorem{theorem}{Theorem}[section]
\newtheorem{coro}[theorem]{Corollary}
\newtheorem{lemma}[theorem]{Lemma}
\newtheorem{prop}[theorem]{Proposition}

\theoremstyle{definition}

\newtheorem{remark}[theorem]{Remark}

\title{$h$-principle for Curves with Prescribed Curvature}
\author[M.~Wasem]{Micha Wasem}
\address{Department of Mathematics, ETH Z\"urich, R\"amistrasse 101,
    8092 Z\"urich, Switzerland}
\email{micha.wasem@math.ethz.ch}
\date{\today}

\usepackage[T2A,T1]{fontenc}
\newcommand{\Che}{\mbox{\usefont{T2A}{\rmdefault}{m}{n}\CYRCH}}

\newcommand{\Ell}{\mbox{\usefont{T2A}{\rmdefault}{m}{n}\CYRL}}

\begin{document}
\maketitle
\begin{abstract}
We prove that every immersed $C^2$-curve $\gamma$ in $\R^n$, $n\geqslant 3$ with curvature $k_{\gamma}$ can be $C^1$-approximated by immersed $C^2$-curves having prescribed curvature $k>k_{\gamma}$. The approximating curves satisfy a $C^1$-dense $h$-principle. As an application we obtain the existence of $C^2$-knots of arbitrary positive curvature in each isotopy class, which generalizes a similar result by McAtee for $C^2$-knots of constant curvature.\end{abstract}
\section{Introduction}

The purpose of this note is to prove that if $n\geqslant 3$, any given immersed $C^2$-curve $\gamma$ in $\R^n$ with curvature $k_{\gamma}$ can be $C^1$-approximated by an immersed $C^2$-curve with prescribed curvature $k$, provided $k>k_{\gamma}$. It is shown that the sufficient condition $k>k_{\gamma}$ is almost optimal in the sense that $k\geqslant k_\gamma$ is necessary. Moreover, if the initial curve is embedded, we show that the initial and the approximating curve can be chosen to be isotopic.\\

Let $\I=[0,b]\subset\R$ be a compact interval. A curve $\gamma\in C^\ell(\I,\R^n),~n\geqslant 3$ is called \emph{closed} if it agrees in $0$ and $b$ to $\ell$ orders. The \emph{curvature} of an immersed curve $\gamma\in C^2(\I,\R^n)$ is defined as $k_{\gamma}:=|\dot {\mathrm T}|/|\dot\gamma|$, where $\mathrm T:=\dot\gamma/|\dot\gamma|$. For $n=3$, this formula reduces to the usual expression $k_{\gamma}=|\dot\gamma\times\ddot\gamma|/|\dot\gamma|^3$. We will tacitly use the identification $S^1\cong \R/2\pi\Z$ and make repeated use of the standard $C^\ell$-norms
$$
\|\gamma\|_{C^{\ell}(\I)}\coloneqq \sup_{t\in\I} \sum_{m\leqslant \ell}\left|\frac{\mathrm d^m\gamma(t)}{\mathrm dt^m}\right|.
$$
A homotopy $\gamma_s\in C^\ell(\I,\R^n),~s\in[0,1]$ is called \emph{regular} respectively an \emph{isotopy}, if $\gamma_s$ is an immersion respectively an embedding for all $s$. We are now ready to state our main result.
\begin{theorem}\label{main}
Let $\gamma_0\in C^2(\I,\R^n), ~n\geqslant 3$ be an immersed (embedded) curve. Then for every $k\in C^\infty(\I)$ satisfying $k>k_{\gamma_0}$ and every $\varepsilon>0$, there exists a regular homotopy (an isotopy) $\gamma_s\in C^2(\I,\R^n)$, $s\in[0,1]$ such that $\|\gamma_0-\gamma_s\|_{C^1(\I)}<\varepsilon$ for all $s$ and such that $k_{\gamma_1}=k$.
\end{theorem}

We will start by observing that the sufficient condition $k>k_{\gamma_0}$ is almost optimal in the sense that $k\geqslant k_{\gamma_0}$ is necessary for the statement to hold:

\begin{prop}
Let $\gamma_i\in C^2(\I,\R^n)$ be a sequence of immersed curves such that $k_{\gamma_i}=k$ and let $\gamma\in C^2(\I,\R^n)$ be an immersion. If $\lim\limits_{i\to \infty}\|\gamma-\gamma_i\|_{C^1(\I)}=0$, then $k_\gamma\leqslant k$.
\end{prop}
\begin{proof}
Let $t_0\in\I$ and let $U$ be a relatively open interval containing $t_0$. From the lower semicontinuity of the integral curvature functional (see \cite[Theorem 5.1.1, p. 120-121]{alexandrov}) with respect to uniform convergence we obtain
$$
\int_{U}k_\gamma(t)\,\mathrm dt \leqslant \liminf_{i\to\infty} \int_{U}k_{\gamma_i}(t)\,\mathrm dt=\int_{U}k(t)\,\mathrm dt.
$$
By arbitrariness of $U$ and continuity of $k_\gamma$ and $k$, it follows that $k_\gamma(t_0)\leqslant k(t_0)$.
\end{proof}

The variant of Theorem \ref{main} for closed curves (see Corollary \ref{knots}) generalizes a result due to McAtee \cite{mcatee}, who proved that there exists a $C^2$ knot of constant curvature in each isotopy class building upon the work of Koch and Engelhardt \cite{engelhardt}, who gave the first explicit construction a non-planar closed $C^2$-curve of constant curvature. Ghomi proved that curves with constant curvature satisfy a relative $C^1$-dense $h$-principle (see \cite{ghomi}) and proved the existence of smooth knots with constant curvature in each isotopy class that are $C^1$-close to a given initial knot. Observe that we do not achieve smooth curves with prescribed curvature, however, in our Theorem, the curvature need not be constant. Both, Ghomi's result and Theorem \ref{main} are -- in the language of Gromov \cite{gromov} and Eliashberg, Mishachev \cite{eliashberg} -- manifestations of a $C^1$-dense $h$-principle.\\

We will prove Theorem \ref{main} directly using a variant of convex integration \`a la Nash and Kuiper. The strategy of the proof consists in reducing the curvature defect $k-k_{\gamma}>0$ successively in steps while keeping careful control on the $C^2$-norm of the resulting maps during the process. 
\subsection*{Acknowledgements}
This article is part of my PhD thesis. I would like to warmly thank my advisor Norbert Hungerb\"uhler for his guidance, support and interest in this work. Furthermore I would like to thank Anand Dessai, Felix Hensel, Berit Singer, Nicolas Weisskopf and Thomas Mettler for helpful comments.
\newpage\section{Step}

The main building block for a step is given by a function whose second derivatives satisfy a circle equation and a suitable estimate:
\begin{lemma}
There exists $\Che\in C^\infty(\R\times S^1,\R^2)$, $(s,t)\mapsto\Che(s,t)$ satisfying
\begin{align}
|\partial_{tt}\Che(s,t)|&\leqslant C|s|\label{c1estimatecurves}\\
(1+\partial_{tt}\Che_1(s,t))^2+(\partial_{tt}\Che_2(s,t))^2 &= 1+s^2\label{circleequation}\end{align}
for some constant $C>0$.

\end{lemma}
\begin{proof}
We will first prove that there exists a function $f\in C^\infty(\R)$ such that
$$J_0(f(s))=\frac{1}{\sqrt{1+s^2}}\eqqcolon w(s),$$
where $J_0$ denotes the zeroth Bessel function of the first kind:
For the local existence around zero (see also \cite[Lemma 2]{delellis}), consider the function
$$
F(s,r)\coloneqq J_0(r^{\frac{1}{2}})-w(s)
$$
Direct computations involving the series expansion of $J_0$ yield $F(0,0)=0$ and $\partial_rF(0,0)=-\frac{1}{4}$. The implicit function Theorem applies and we obtain $\delta>0$ and a function $h\in C^{\infty}((-\delta,\delta))$ with $h(0)=0$ such that $F(s,h(s))=0.$ Moreover, since $\partial_sF(0,0)=0$ and $\partial_s^2F(0,0)=1$, we find $h'(0)=0$ and $h''(0)=4$. Hence we get a smooth $f(s)$ satisfying $f^2(s)=h(s)$. Direct computations yield $f(0)=0$ and $f'(0)=\sqrt{2}$.
In order to prove global existence, observe that $w(s)\coloneqq (1+s^2)^{-\sfrac{1}{2}}$ is smooth and takes values in $(0,1]$. Since $J_0:[0,\mu]\to[0,1]$ is a bijection ($\mu$ being the smallest positive zero of $J_0$) and $J'_0$ doesn't admit any zero on $(0,\mu]$, its inverse $J_0^{-1}$ is in $C^{\infty}([0,1))$. Now set
$$
f(s)\coloneqq \operatorname{sgn}s\cdot J_0^{-1}\left(\frac{1}{\sqrt{1+s^2}}\right).
$$
This function is clearly smooth on $\R\setminus\{0\}$. Since $f$ corresponds around zero to the function constructed by means of the implicit function Theorem, $f\in C^\infty(\R)$.
With this choice of $f$, let $\Gamma:\R^2\to\R^2$,
$$
\Gamma(s,t)\coloneqq \int_0^t \left(\sqrt{1+s^2}\begin{pmatrix}\cos (f(s)\sin u)\\ \sin(f(s)\sin u)\end{pmatrix}-\begin{pmatrix}1\\ 0\end{pmatrix}\right)\mathrm du.
$$
$\Gamma$ is $2\pi$-periodic in the second argument, hence $\Gamma:\R\times S^1\to \R^2$, and it holds that
\begin{equation}
|\partial_t \Gamma(s,t)|\leqslant C|s|,\label{c1-estimate}
\end{equation}
for some constant $C>0$. For the periodicity, we compute
$$\begin{aligned}
\Gamma(s,t+2\pi)-\Gamma(s,t) & = \int_t^{t+2\pi} \left(\sqrt{1+s^2}\begin{pmatrix}\cos (f(s)\sin u)\\ \sin(f(s)\sin u)\end{pmatrix}-\begin{pmatrix}1\\ 0\end{pmatrix}\right)\mathrm du\\
& = \int_0^{2\pi} \left(\sqrt{1+s^2}\begin{pmatrix}\cos (f(s)\sin u)\\ \sin(f(s)\sin u)\end{pmatrix}-\begin{pmatrix}1\\ 0\end{pmatrix}\right)\mathrm du\\
& =2\pi\begin{pmatrix}\sqrt{1+s^2}J_0(f(s))-1 \\ 0\end{pmatrix}=\begin{pmatrix}0\\0\end{pmatrix}.
\end{aligned}$$
In order to prove \eqref{c1-estimate}, observe that since $\partial_t\Gamma(0,t)=0$, integrating in $s$ yields
$$\partial_t\Gamma(s,t)=\int_0^s\partial_s\partial_t\Gamma(r,t)\mathrm dr,$$
and we need to show that $|\partial_s\partial_t\Gamma|$ is bounded by some constant $C$. We compute
$$
\partial_s\partial_t\Gamma(s,t) =\frac{s}{\sqrt{1+s^2}}\begin{pmatrix}\cos(f(s)\sin t)\\ \sin(f(s)\sin t)\end{pmatrix}+\sqrt{1+s^2}f'(s)\sin t\begin{pmatrix}-\sin(f(s)\sin t)\\ \cos(f(s)\sin t)\end{pmatrix},
$$
hence
$$\begin{aligned}
|\partial_s\partial_t\Gamma(s,t)|^2& \leqslant \frac{s^2}{1+s^2}+(1+s^2)[f'(s)]^2\\
& \leqslant 1+\left(\frac{f'(s)}{w(s)}\right)^2.
\end{aligned}
$$
We claim that the second term vanishes as $s\to\pm\infty$. Recall that
$$|f'(s)|=|(J_0^{-1})'(w(s))w'(s)|$$
for $s\ne 0$. We get
$$
\left|\frac{f'(s)}{w(s)}\right| = \left| \left(J_0^{-1}\right)'(w(s))\frac{w'(s)}{w(s)}\right|.
$$
Now $\mu$ (the smallest positive zero of $J_0$) is a simple zero (see for example \cite{stegun} p. 370), thus $\lim\limits_{s\to\pm\infty}\left|\left(J_0^{-1}\right)'(w(s))\right|=\left|\frac{1}{J_0'(\mu)}\right|<\infty$ and the claim follows from
$$
\lim_{s\to\pm\infty}\frac{w'(s)}{w(s)}=0.
$$
This implies that $|\partial_s\partial_t\Gamma|$ is bounded by a constant and we obtain \eqref{c1-estimate}.
Let now
$$\Che(s,t)\coloneqq \int_0^t\Gamma(s,u)\,\mathrm du -\frac{t}{2\pi}\int_0^{2\pi}\hspace{-0.1cm}\Gamma(s,u)\,\mathrm du,$$
By construction, $\Che\in C^\infty(\R\times S^1,\R^2)$ has all the desired properties.
\end{proof}\newpage
\begin{prop}[$i$-th Step]\label{stepcurves}
Let $\gamma_{i-1}\in C^\infty(\I,\R^n)$, $n\geqslant 3$ be an immersed curve and let $k\in C^\infty(\I)$ such that $0<k_{\gamma_{i-1}}<k$. Then there exists an immersed curve $\gamma_{i}\in C^\infty(\I,\R^n)$ satisfying
\begin{align}
\|\gamma_{i}-\gamma_{i-1}\|_{C^1(\I)} & <\varepsilon,\label{curvesc1}\\
\left\|\ddot\gamma_{i}-\ddot\gamma_{i-1}\right\|_{C^0(\I)}& < C\|\dot\gamma_{i-1}\|^2_{C^0(\I)}\|k^2 - k_{\gamma_{i-1}}^2\|^{\sfrac{1}{2}}_{C^0(\I)},\label{curvesc2}\\
\|k^2-k_{\gamma_{i}}^2\|_{C^0(\I)}& <\varepsilon,\label{curvaturec0}\\
0 <k_{\gamma_i}&<k,\label{curvatureok}
\end{align}
where $C$ is the constant from estimate \eqref{c1estimatecurves}.
\end{prop}
\begin{proof}
First consider the map $\varphi:\I\to\mathrm J$ given by
$$
t\mapsto \int_0^t|\dot\gamma_{i-1}(u)|\,\mathrm du
$$
and consider $\bar\gamma_{i-1}\coloneqq \gamma_{i-1}\circ\varphi^{-1}\in C^\infty(\mathrm J,\R^n)$ and $\bar k\coloneqq k\circ\varphi^{-1}\in C^\infty(\mathrm J)$, where $\mathrm J=[0,\varphi(b)]$. Let $\lambda>0$, $0<\delta<1$, $\xi = \ddot{\bar\gamma}_{i-1}$ and $\zeta\in(\operatorname{span}\{\dot{\bar\gamma}_{i-1},\xi\})^\perp$ such that $|\zeta|=|\xi|$. The vector field $\zeta$ can be chosen to be closed if $\gamma_{i-1}$ is closed. Indeed, the restriction of the (trivial) tangent bundle of $\R^n$ onto the closed curve is trivial and $\dot\gamma$ and $\xi$ are two orthogonal sections in this bundle. Hence there exist $n-2$ additional linearly independent sections that are orthogonal to $\operatorname{span}\{\dot{\bar\gamma}_{i-1},\xi\}$. If $n=3$, one can take $\zeta=\dot{\bar\gamma}_{i-1}\times \xi$. For
$$
a(t)\coloneqq\sqrt{(1-\delta)\left(\frac{\bar k^2(t)}{|\xi(t)|^2}-1\right)},
$$
let
\begin{equation}\label{ansatzmitckurven}
\bar{\gamma}_{i}(t)=\bar\gamma_{i-1}(t)+\frac{1}{\lambda^2}\bigg[\Che_1(a(t),\lambda t)\xi(t)+\Che_2(a(t),\lambda t)\zeta(t)\bigg].
\end{equation}
The first and the second derivative of $\bar{\gamma}_i$ are given by
$$
\begin{aligned}
\dot{\bar{\gamma}}_{i} & =\dot{\bar\gamma}_{i-1}+\frac{1}{\lambda}\bigg[\partial_t\Che_1(a,\lambda \cdot)\xi+\partial_t\Che_2(a,\lambda \cdot)\zeta\bigg]+O(\lambda^{-2})\\
\ddot{\bar{\gamma}}_{i} & =\ddot{\bar\gamma}_{i-1}+\partial_{tt}\Che_1(a,\lambda \cdot)\xi+\partial_{tt}\Che_2(a,\lambda \cdot)\zeta+O(\lambda^{-1}).
\end{aligned}
$$
This implies $\|\bar{\gamma}_{i}-\bar\gamma_{i-1}\|_{C^0(\mathrm J)}=O(\lambda^{-2})$ and $\|\bar{\gamma}_{i}-\bar\gamma_{i-1}\|_{C^1(\mathrm J)}=O(\lambda^{-1})$. Setting $\gamma_i = \bar\gamma_i\circ\varphi$ we obtain \eqref{curvesc1} from choosing $\lambda\gg 1$ and from
$$
\|\gamma_{i}-\gamma_{i-1}\|_{C^1(\I)}\leqslant \left(O(\lambda^{-1})+\|\dot\varphi\|_{C^0(\I)}\right)O(\lambda^{-1}).
$$
Since $|\xi|=k_{\bar\gamma_{i-1}}$, using \eqref{c1estimatecurves} and a suitable choice of $\lambda$ yields
$$\begin{aligned}
|\ddot{\bar{\gamma}}_{i}-\ddot{\bar\gamma}_{i-1}|^2 & = |\xi|^2|\partial_{tt}\Che(a,\lambda\cdot)|^2+O(\lambda^{-1})\\
& \leqslant C^2|\xi|^2a^2+O(\lambda^{-1})\\
& \leqslant C^2(1-\delta)(\bar k^2 - |\xi|^2)+O(\lambda^{-1})\\
& \leqslant C^2(1-\tfrac{\delta}{2})(\bar k^2 - |\xi|^2)
\end{aligned}$$
and -- since curvature is parameter-invariant -- this implies the estimate \eqref{curvesc2}
$$\begin{aligned}
\|\ddot\gamma_{i}-\ddot\gamma_{i-1}\|_{C^0(\I)} & \leqslant C\sqrt{1-\tfrac{\delta}{2}}\|\dot\varphi\|^2_{C^0(\I)}\|k^2 - k_{\gamma_{i-1}}^2\|_{C^0(\I)}^{\sfrac{1}{2}}+\|\ddot\varphi\|_{C^0(\I)}O(\lambda^{-1})\\
& \leqslant C\|\dot\gamma_{i-1}\|^2_{C^0(\I)}\|k^2 - k_{\gamma_{i-1}}^2\|_{C^0(\I)}^{\sfrac{1}{2}}.
\end{aligned}$$
Using \eqref{circleequation} we obtain \eqref{curvaturec0} and \eqref{curvatureok} from
$$\begin{aligned}
\bar k^2 - k_{\bar{\gamma}_{i}}^2 & = \bar k^2 - |\ddot{\bar\gamma}_{i}-\dot{\bar\gamma}_{i}\langle\dot{\bar\gamma}_{i},\ddot{\bar\gamma}_{i}\rangle|^2 + O(\lambda^{-1}) \\
& =\bar k^2 - \left[(1+\partial_{tt}\Che_1(a,\lambda\cdot))^2+(\partial_{tt}\Che_2(a,\lambda\cdot))^2\right]|\xi|^2+O(\lambda^{-1})\\
& =\bar k^2 - (1+a^2)|\xi|^2+O(\lambda^{-1})\\
& = \delta(\bar k^2 - k_{\bar\gamma_{i-1}}^2)+O(\lambda^{-1})
\end{aligned}$$
provided $\lambda \gg 1$ and $\delta\ll 1$. We can ensure that $\gamma_{i}$ is an immersion, since $\gamma_i$ and $\gamma_{i-1}$ can be made arbitrarily $C^1$-close and the immersions form an open set in $C^1(\I,\R^n)$. In the same way we can ensure that $\gamma_i$ is an embedding if $\gamma_{i-1}$ is an embedding. If $\gamma_{i-1}$ is closed, $\gamma_i$ is closed if $\lambda\in({2\pi}/{\varphi(b)})\N$ (Fig. \ref{bild}).
\end{proof}
\begin{figure}
\begin{center}
\includegraphics[scale=0.7]{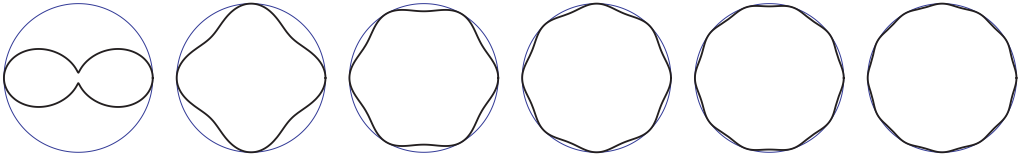}
\caption{Illustration of the ``top view'' of $\bar\gamma_i$ if $\bar\gamma_{i-1}$ is a unit circle, $a(t)=\sqrt{8}$ and $\lambda=1,2,\ldots,6$.}
\label{bild}\end{center} 
\end{figure}
\begin{remark}
The foregoing proof might be simplified if $n\geqslant 4$. In this case, there exist two mutually orthogonal vector fields $\zeta_1$ and $\zeta_2$ of length $|\xi|$ such that $\zeta_1,\zeta_2\in(\operatorname{span}\{\dot{\bar\gamma}_{i-1},\xi\})^\perp$. If $\gamma_{i-1}$ is closed, $\zeta_1$ and $\zeta_2$ can also be chosen to be closed. Then the formula \eqref{ansatzmitckurven} involving the rather complicated function $\Che$ can be replaced by a modified \emph{Nash Twist} (used by Nash in the proof of his celebrated $C^1$-isometric embedding theorem \cite{nash2}):
$$\bar{\gamma}_{i}(t)=\bar\gamma_{i-1}(t)+\frac{a(t)}{\lambda^2}\bigg[\cos(\lambda t)\zeta_1(t)+\sin(\lambda t)\zeta_2(t)\bigg].
$$
In this respect, formula \eqref{ansatzmitckurven} is an analogue to the \emph{Strain}, Kuiper used in \cite{kuiper} in order to lower the codimension requirement in the Nash-Kuiper theorem.
\end{remark}

\section{Iteration and Application to Knot Theory}
Using a $C^2$-perturbation we may assume that the curvature $k_{\gamma_0}$ of the initial curve $\gamma_0$ in Theorem \ref{main} is never zero. This is the content of the following Lemma which is due to Ghomi \cite[Lemma 5.3]{ghomi}, but we will give the proof in our slightly different setting:
\begin{lemma}[Ghomi]\label{genericity}
Let $\gamma\in C^2(\I,\R^n),~n\geqslant 3$ and let $k\in C^0(\I)$ satisfy $k>k_{\gamma}$. Then for every $\varepsilon>0$ there exists $\hat\gamma\in C^2(\I,\R^n)$ satisfying $k>k_{\hat\gamma}>0$ and $\|\gamma-\hat\gamma\|_{C^2(\I)}<\varepsilon$.
\end{lemma}
\begin{proof}
Observe that the curvature of $\gamma$ vanishes if and only if $\dot\gamma$ and $\ddot\gamma$ are linearly dependent. Consider therefore the set
$$\Ell\subset J^2(\I,\R^n)\cong \I\times\R^n\times\R^n\times\R^n$$
containing all the elements of $J^2(\I,\R^n)$ having linearly dependent last two components. Observe that $\dim\Ell=2n+2$ and $\dim (J^2(\I,\R^n))=3n+1$. By Thom's transversality Theorem, there exists a dense set of curves $\hat\gamma\in C^2(\I,\R^n)$ such that $j^2\hat\gamma$ is transversal to $\Ell$. Since $n\geqslant 3$, we obtain
$$\dim \I+\dim\Ell = 2n+3 < 3n+1=\dim (J^2(\I,\R^n))$$
 and we conclude that $j^2\hat\gamma$ and $\Ell$ do not intersect. By density we can choose $\hat\gamma$ in such a way that $\|\gamma-\hat\gamma\|_{C^2(\I)}<\varepsilon$. This inequality implies that we can also obtain $\|k_{\gamma}-k_{\hat\gamma}\|_{C^0(\I)}<\varepsilon$ and hence we can choose $\hat\gamma$ satisfying $k>k_{\hat\gamma}>0$.\end{proof}

\begin{proof}[Proof of Theorem \ref{main}]
In view of Lemma \ref{genericity} we can assume that $\gamma_0$ has non-vanishing curvature. Using standard regularization we can assume in addition that $\gamma_0\in C^\infty(\I,\R^n)$, since the second derivative of the mollification of $\gamma_0$ can be made arbitrarily close to $\ddot\gamma_0$. After these modifications, $\gamma_0$ is still an immersion (embedding), since immersions and embeddings form an open set in $C^1(\I,\R^n)$ and we can use Proposition \ref{stepcurves} iteratively starting with $\gamma_0$ and a sequence $\varepsilon_i\to 0$ such that
\begin{equation}\label{summenepsilon}\sum_{i=1}^\infty\varepsilon_i=\varepsilon\text{ and }\sum_{i=1}^\infty\sqrt{\varepsilon_i}<\infty.\end{equation}
This implies the uniform estimate
$$\begin{aligned}
\|\dot\gamma_{j}\|_{C^0(\I)}&\leqslant\|\dot\gamma_0\|_{C^0(\I)}+\sum_{i=1}^j\|\dot\gamma_{i}-\dot\gamma_{i-1}\|_{C^0(\I)}\leqslant  \|\dot\gamma_0\|_{C^0(\I)}+\varepsilon
\end{aligned}$$
and hence we find for $j>i$ using \eqref{curvesc1}, \eqref{curvesc2} and \eqref{summenepsilon}:
$$\begin{aligned}
\|\gamma_j-\gamma_i\|_{C^2(\I)}&Ê\leqslant \sum_{\ell=i+1}^j\|\gamma_{\ell}-\gamma_{\ell-1}\|_{C^2(\I)}\\
& \leqslant \sum_{\ell=i+1}^\infty\varepsilon_\ell + C(\|\dot\gamma_0\|_{C^0(\I)}+\varepsilon)^2\sum_{\ell=i+1}^\infty\sqrt{\varepsilon_{\ell-1}}\stackrel{j,i\to\infty}{\longrightarrow}0.\end{aligned}$$
The sequence $(\gamma_i)_{i\in\N}$ is thus Cauchy in $C^2(\I,\R^n)$ and the estimate \eqref{curvaturec0} implies that the limit map $\widetilde\gamma$ has curvature $k$. Using \eqref{curvesc1} and \eqref{summenepsilon} we find that the homotopy $\gamma_s = (1-s)\gamma+s\widetilde\gamma$ satisfies
$$
\|\gamma_s-\gamma\|_{C^1(\I)}\leqslant s\|\widetilde\gamma-\gamma\|_{C^1(\I)}\leqslant \varepsilon.
$$
This concludes the proof since embeddings and immersions form an open set in $C^1(\I,\R^n)$.
\end{proof}
\begin{remark}
The statement about the isotopy in Theorem \ref{main} becomes relevant only if $\gamma_0$ is a knot in $\R^3$. In this case -- and since Proposition \ref{stepcurves} may produce closed curves out of closed curves -- we obtain as an application of Theorem \ref{main}\end{remark}
\begin{coro}\label{knots}
Let $\gamma_0\in C^2(S^1,\R^3)$ be a knot. Then for any positive $k\in C^\infty(S^1)$, there exists a knot $\gamma_1\in C^2(S^1,\R^3)$ in the same isotopy class with curvature $k$.
\end{coro}

\begin{proof}
Observe that for $c>0$ it holds that $k_{c\cdot\gamma_0}=k_{\gamma_0}/c$ and $\gamma_0$ and $c\cdot\gamma_0$ are isotopic. Hence we can choose $c$ in such a way that 
$$\|k (c\cdot\gamma_0)\|_{C^0(S^1)}<\min_{t\in S^1}k(t)$$
and apply Theorem \ref{main}.
\end{proof}

\begin{remark}
By choosing $k$ in Corollary \ref{knots} to be constant, we recover the main result of \cite{mcatee} and a $C^2$-version of the main result of \cite{ghomi}.\end{remark}

\end{document}